\documentclass[a4paper, reqno]{amsart}

\usepackage[all]{xy}
\usepackage[english]{babel}
\usepackage[utf8]{inputenc}
\usepackage{amssymb,amsmath,amsthm}
\usepackage{amsfonts}
\usepackage{graphicx}
\usepackage{psfrag}
\usepackage[dvipsnames]{xcolor}
\usepackage[hidelinks]{hyperref}
\usepackage{csquotes}
\usepackage[alphabetic]{amsrefs}
\usepackage{enumitem}
\usepackage{cancel}
\usepackage[normalem]{ulem}

\allowdisplaybreaks

\newcommand{\Ric}{\mathrm{Ric}}

\newcommand{\e}{\epsilon}

\renewcommand{\S}{\Sigma}

\renewcommand{\Im}{\mathrm{Im} \,}

\newcommand{\E}{\mathcal E}
\newcommand{\A}{\mathcal A}

\renewcommand{\o}{\omega}
\renewcommand{\O}{\mathcal O}
\newcommand{\supp}{\mathrm{supp}}

\renewcommand{\a}{\alpha}
\newcommand{\de}{\delta}
\renewcommand{\b}{\beta}
\renewcommand{\d}{\partial}

\newcommand{\abs}[1]{\left\lvert#1\right\rvert}

\renewcommand{\H}{\mathrm H}
\newcommand{\Ho}{\mathcal H}
\renewcommand{\E}{\mathcal E}

\newcommand{\U}{\mathcal U}

\newcommand{\md}{\mathrm d}
\newcommand{\p}{\mathrm p}
\renewcommand{\q}{\mathrm q}

\newcommand{\Char}{\mathrm{Char}}
\newcommand{\s}{\sigma}

\newcommand{\Diff}{\mathrm{Diff}}
\newcommand{\bl}{{\mathrm{b}}}

\newcommand{\Diffb}{\Diff_\bl}

\newcommand{\sgn}{\operatorname{sgn}}

\theoremstyle{plain}
\newtheorem{thm}{Theorem}[section]
\newtheorem{prop}[thm]{Proposition}
\newtheorem{prop: main}[thm]{Main Lemma}
\newtheorem{lemma}[thm]{Lemma}

\theoremstyle{definition}
\newtheorem{definition}[thm]{Definition}

\newtheorem{remark}[thm]{Remark}

\newtheorem{assumption}[thm]{Assumption}

\newcommand{\R}[0]{\mathbb{R}}						
\newcommand{\C}[0]{\mathbb{C}}						

\makeatletter
\@namedef{subjclassname@1991}{Subject}
\@namedef{subjclassname@2000}{Subject}
\@namedef{subjclassname@2010}{2020 Mathematics Subject Classification}
\makeatother

\title[Wave equations in the subextremal Kerr-de Sitter spacetime]{Wave equations in the Kerr-de Sitter spacetime: the full subextremal range}

\author{Oliver Petersen}
\address{Department of Mathematics, KTH Stockholm, Lindstedtsvägen 25, 11428 Stockholm, Sweden}
\email{oliverlp@kth.se}

\author{Andr\'{a}s Vasy}
\address{Department of Mathematics, Stanford University, CA 94305-2125, USA}
\email{andras@math.stanford.edu}


\begin{document}
\hbadness=100000
\vbadness=100000

\begin{abstract}
We prove that solutions to linear wave equations in a subextremal
Kerr-de Sitter spacetime have asymptotic expansions in quasinormal modes up to a decay order given by the normally hyperbolic trapping, extending \cite{V2013}.
The main novelties are a different way of obtaining a Fredholm setup
that defines the quasinormal modes and a new analysis of the trapping of lightlike geodesics in the Kerr-de Sitter spacetime, both of which apply in the full subextremal range.
In particular, this reduces the question of decay for solutions to
wave equations to the question of mode stability.
\end{abstract}

\maketitle

\tableofcontents
\begin{sloppypar}

\section{Introduction}
Kerr-de Sitter metrics are Lorentzian metrics (which we take of
signature $(-,+,+,+)$) on
$\R_{t_*}\times(0,\infty)_r\times S^2$ solving Einstein's equation
$\Ric(g)=\Lambda g$ with a
cosmological constant $\Lambda>0$.  They depend on three parameters:
apart from the cosmological constant $\Lambda$, these are the mass
$m>0$ and the angular
momentum $a
\in \R$ of the black hole the metric corresponds to. Here the `black
hole' nature corresponds to the presence of certain null-hypersurfaces,
called horizons, which for the Kerr-de Sitter metrics lie at certain
values of $r$. The metric is given in terms of a quartic polynomial $\mu$, and
the three parameters then have to
satisfy an additional condition, namely that $\mu$ has four distinct
real roots; the significance of this condition is due to the horizons
lying at the roots of $\mu$. Kerr-de Sitter metrics corresponding to
these parameters are called {\em subextremal}.

In \cite{V2013} the wave equation on Kerr-de Sitter spacetimes was
analyzed for a large but not complete range of the parameters by
showing that it fits into the Fredholm framework based on microlocal
analysis developed there; the result was an expansion of solutions of
the wave equation into a finite number of terms, corresponding to
{\em quasinormal modes} described below in Section~\ref{subsec:result}, modulo an
exponentially decaying tail. 
That paper then formed the basis of the
approach to non-linear wave equations by Hintz and Vasy \cites{HV2015,HV2016} which culminated in the proof of the stability of slowly rotating Kerr-de Sitter black holes in \cite{HV2018}; the first black hole stability result without symmetry assumptions. 
The purpose of this paper is to complete \cite{V2013} by extending its results to the full subextremal range. 
In particular, we show, in the full subextremal range, the expansion of solutions of the wave
equation into a finite number of terms, given by quasinormal modes,
modulo an exponentially decaying tail, see Theorem~\ref{thm: asymptotic expansion}. While for linear equations the question
whether any of these quasinormal mode terms is non-decreasing is
irrelevant, it becomes highly relevant for non-linear wave equations,
where the absence (perhaps apart from certain well-understood ones) of quasinormal modes corresponding to non-decaying
terms is called {\em mode stability}, see Section~\ref{subsec:result}.

It turns out that the additional limitations of \cite{V2013} had two
separate origins. 
The basic approach of \cite{V2013} was to Fourier transform the wave operator along the Killing vector field $-\partial_{t_*}$ to obtain a family of operators, $P_\s$, depending on the Fourier dual variable $\s$. 
These operators are non-elliptic, hence their analysis involves the Hamilton flow in the characteristic set of the principal symbol. 
One of the additional assumptions in \cite{V2013}, denoted by (6.13) there, was to ensure that for each $\s \in \C$, $P_\s$ satisfies a non-trapping condition; it was shown there that indeed in the limiting case of (6.13), trapping appears. 
It was shown that under this {\em classical non-trapping} condition the family $P_\s$ is Fredholm for each $\s$ with the analytic Fredholm theorem applicable. 
The other additional assumption of \cite{V2013} concerned the behavior of the family $P_\s$ for $\s$ with large real part and bounded imaginary part, namely in this large parameter (or as rescaled there, semiclassical) sense, which is stronger than the fixed $\s$ consideration, the only trapping is normally hyperbolic trapping; this corresponds to the photon sphere when $a=0$.
This {\em semiclassical normally hyperbolic trapping} assumption was shown under the additional condition (6.27) in \cite{V2013}, namely
that $|a|<\frac{\sqrt{3}}{2}m$. 
We recall that, with a basic microlocal analysis background, the analysis of \cite{V2013} was self-contained apart from using the normally hyperbolic trapping analytic results (estimates for a microlocalized at the trapping version of $P_\s^{-1}$) of Wunsch and Zworski \cite{WZ2011}. 
Normally hyperbolic trapping has since been
investigated in numerous papers \cites{Dya2015b, Dya2016,NZ2015,H2021}, giving a more precise
version of \cite{WZ2011}, but for the present purposes \cite{WZ2011}
still suffices.

In this paper we remove these limitations in two steps. First, we change the Killing vector field with respect to which we perform the Fourier transform. 
It turns out that for an appropriate choice of the Killing vector field, the Fourier transformed operator $P_\s$ is classically non-trapping in the full subextremal range. 
Second, we show that in fact the semiclassical normally hyperbolic trapping holds in the full subextremal range as well.
Since these were the only two additional limitations in \cite{V2013} in the
subextremal range, this immediately implies that all of the results of
\cite{V2013} in fact hold in the full subextremal range.

While in this paper we focus on the scalar wave equation, the changes
when turning to tensorial wave equations
in general are minor, as shown earlier already in \cite{HV2018},
mainly affecting certain threshold quantities. We will discuss this
elsewhere.

As a comparison, we mention that it has been known for some time that
the full subextremal range for the vanishing cosmological constant
($\Lambda=0$) Kerr spacetime, $|a|<m$, behaves in the same way as for
small $|a|$, see \cites{FKSY06, DR2011, Dya2015, SR2015, DRSR16} and references
therein, and there even mode stability is known, see
\cites{W1989,SR2015,CTC2021} and references therein.

\subsection{Kerr-de Sitter spacetimes}
We now describe the subextremal condition in more detail.
The polynomial $\mu$ is given by
\begin{equation} \label{eq: mu}
	\mu(r) := \left(r^2 + a^2\right)\left(1 -\frac{\Lambda r^2}3\right) - 2mr,
\end{equation}
and the subextremality condition is that it has four distinct real roots 
\[
	r_- < r_C < r_e < r_c,
\]
which is equivalent to the discriminant condition
\begin{equation} \label{eq: discriminant}
	- \left( 1 + \frac{\Lambda a^2}3 \right)^4 \left( \frac a m \right)^2 + 12 \left( 1 - \frac{\Lambda a^2}3 \right) \Lambda a^2 + \left(1 - \frac{\Lambda a^2}3 \right)^3 - 9 \Lambda m^2 > 0.
\end{equation}
It follows that there is a unique $r_0 \in (r_e, r_c)$, such that
\begin{equation} \label{eq: r_0}
	\mu'(r_0) = 0.
\end{equation}
As we will see, $r_0$ will play a crucial role in choosing the Killing
vector field for the Fourier transform.
The domain of outer communication $M$ in the subextermal Kerr-de Sitter spacetime is given (in Boyer-Lindquist coordinates) by the real analytic spacetime
\[
	M := \R_t \times (r_e, r_c)_r \times S^1_\phi \times (0, \pi)_\theta,
\]
with real analytic metric
\begin{align*}
	g
		&= (r^2 + a^2 \cos^2(\theta))\left( \frac{\md r^2}{\mu(r)} + \frac{\md \theta^2}{c(\theta)}\right) \\*
		&\quad + \frac{c(\theta)\sin^2(\theta)}{b^2\left(r^2 + a^2 \cos^2(\theta)\right)}\left(a \md t - \left(r^2 + a^2\right)\md \phi\right)^2 \\*
		&\quad - \frac{\mu(r)}{b^2\left(r^2 + a^2 \cos^2(\theta)\right)}\left(\md t - a \sin^2(\theta)\md \phi\right)^2,
\end{align*}
where
\[
	b := 1 + \frac{\Lambda a^2}3, \quad c(\theta) := 1 + \frac{\Lambda a^2}3 \cos^2(\theta).
\]
One easily verifies that this metric extends real analytically to the north and south poles $\theta = 0, \pi$.

The Boyer-Lindquist coordinates are singular at the roots of $\mu$.
We therefore extend this metric real analytically over the future event horizon and the future cosmological horizon, corresponding to the roots $r = r_e$ and $r = r_c$, respectively.
One way to do this is by the following coordinate change:
\begin{equation}
	\begin{split} \label{eq: extended coordinates}
	t_*
		&:= t - \Phi(r), \\
	\phi_*
		&:= \phi - \Psi(r),
	\end{split}
\end{equation}
where $\Phi$ and $\Psi$ satisfy
\begin{align*}
		\Phi'(r)
			&= b \frac{r^2 + a^2}{\mu(r)} f(r), \\
		\Psi'(r)
			&= b \frac a{\mu(r)} f(r),
\end{align*}
where
\[
	f: (r_e - \de, r_c + \de) \to \R,
\]
is a real analytic function such that
\begin{equation} \label{eq: f cond}
	f(r_e) 
		= -1, \quad f(r_c) = 1.
\end{equation}
The new form of the metric is
\begin{equation}
\begin{split}
	g_*
		&= (r^2 + a^2 \cos^2(\theta))\frac{1 - f(r)^2}{\mu(r)} \md r^2 \\*
		&\qquad - \frac 2 b f(r) (\md t_* - a \sin^2(\theta) \md \phi_*)\md r \\*
		&\qquad - \frac{\mu(r)}{b^2\left(r^2 + a^2 \cos^2(\theta)\right)}\left(\md t_* - a \sin^2(\theta) \md \phi_* \right)^2 \\*
		&\qquad + \frac{c(\theta)\sin^2(\theta)}{b^2\left(r^2 + a^2 \cos^2(\theta)\right)}\left(a\md t_* - \left(r^2 + a^2\right) \md \phi_*\right)^2 \\*
		&\qquad + (r^2 + a^2 \cos^2(\theta))\frac{\md \theta^2}{c(\theta)},
\end{split}
\label{eq: metric at horizons}
\end{equation}
which extends real analytically to 
\[
	\R_{t_*} \times (r_e - \de, r_c + \de)_r \times S^2_{\phi_*, \theta}.
\]
The two real analytic lightlike hypersurfaces
\begin{align*}
	\Ho_e^+ 
		&:= \R_{t_*} \times \{r_e\} \times S^2_{\phi_*, \theta}, \\
	\Ho_c^+
		&:= \R_{t_*} \times \{r_c\} \times S^2_{\phi_*, \theta},
\end{align*}
are called the \emph{future event horizon} and \emph{future cosmological horizon}, respectively.
Note that the real analytic Killing vector fields $\d_t$ and $\d_\phi$, in Boyer-Lindquist coordinates, extend to real analytic Killing vector fields $\d_{t_*}$ and $\d_{\phi_*}$ over the horizons.

\begin{remark} \label{rmk: t star hypersurfaces}
We claim that there is a real analytic choice of $f$, such that the constant $t_*$ hypersurfaces are everywhere spacelike.
It is straightforward to check that this is equivalent to
\begin{equation} \label{eq: dt star timelike}
	1 - \frac{a^2 \mu(r)}{(r^2 + a^2)^2} 
		> f(r)^2.
\end{equation}
The left hand side in \eqref{eq: dt star timelike} is indeed positive, since $\mu(r) < r^2 + a^2$ for $r > 0$.
Let 
\[
	f_0: (r_e - \de, r_c + \de) \to \R
\]
be any real analytic function satisfying \eqref{eq: f cond} and write
\[
	f(r)
		= f_0(r) + \mu(r) f_1(r).
\]
The function $f$ satisfies \eqref{eq: f cond} for any choice of real analytic function 
\[
	f_1: (r_e - \de, r_c + \de) \to \R.
\]
Define
\[
	f_\pm(r) 
		:= - \frac{f_0(r)}{\mu(r)} \pm \frac1{\mu(r)} \sqrt{1 - \frac{a^2 \mu(r)}{(r^2 + a^2)^2}}
\]
for any $r \notin \{r_e, r_c\}$.
Note that $f_-$ has a continuous extension on $(r_e-\de, r_c)$ and $f_+$ has a continuous extension on $(r_e, r_c+\de)$ and
\[
	\lim_{r \downarrow r_e} f_+(r) = \infty = - \lim_{r \uparrow r_c} f_-(r).
\]
The condition \eqref{eq: dt star timelike} becomes 
\begin{equation} \label{eq: f 1 condition}
	f_1(r) 
		\in 
		\begin{cases}
			(f_-(r), \infty), & r \in (r_e-\de, r_e], \\
			(f_-(r), f_+(r)), & r \in (r_e, r_c), \\
			(-\infty, f_+(r)), & r \in [r_c, r_c+ \de).
		\end{cases}
\end{equation}
There are many real analytic functions $f_1$ satisfying this condition, making sure that the constant $t_*$ hypersurfaces are everywhere spacelike.
\end{remark}

\subsection{Main result}\label{subsec:result}

\begin{assumption} \label{ass: main} \
\begin{itemize}
\item Let $(M_*, g_*)$ be a subextremal Kerr-de Sitter spacetime, extended over the future event horizon and the future cosmological horizon, where
\[
	M_* := \R_{t_*} \times (r_e - \de, r_c + \de)_r \times S^2_{\phi_*, \theta},
\]
with $\de > 0$ small enough so that the boundary hypersurfaces
\[
	\R_{t_*} \times \{r_e - \de\} \times S^2_{\phi_*, \theta}, \quad  \R_{t_*} \times \{r_c + \de\} \times S^2_{\phi_*, \theta}
\]
are \emph{spacelike} and with $f$ chosen as in Remark \ref{rmk: t star hypersurfaces} so that the hypersurfaces
\[
	\{t_* = c \} \times (r_e - \de, r_c + \de)_r \times S^2_{\phi_*, \theta}
\]
are \emph{spacelike}, for all $c \in \R$.
\item Let $A$ be a smooth complex valued function on $M_*$ such that
\[
	\d_{t_*} A = \d_{\phi_*} A = 0.
\]
We let $P$ be the linear wave operator given by
\[
	P = \Box + A.
\]
\end{itemize}
\end{assumption}

\noindent
For any subset $\U \subset M_*$, we use the notation $C^\infty(\U)$ and $C^\o(\U)$ for the smooth and real analytic complex functions on $\U$, respectively.

\subsubsection{Quasinormal modes}

One of the main novelties in this paper is a new definition of quasinormal modes. 
More precisely, we define the quasinormal modes with respect to a different Killing vector field than in previous literature.
As mentioned above, there is a unique $r_0 \in (r_e, r_c)$ such that
\[
	\mu'(r_0) = 0.
\]

\begin{definition}[Quasinormal mode]
A complex function
\[
	u \in C^\infty(M_*)
\]
is called a {\em quasinormal mode}, with {\em quasinormal mode frequency} $\s \in \C$, if
\[
	\left(\d_{t_*} + \frac a{r_0^2 + a^2} \d_{\phi_*}\right) u = - i \s u
\]
and 
\[
	P u = 0.
      \]
Quasinormal modes and mode frequencies are also called resonant states and resonances.
\end{definition}

This certain choice of $r_0$ is important in order to get a Fredholm theory for the induced mode equation, which applies in the full subextremal range.

\begin{remark}\label{rem:quasi}
We can write any quasinormal mode as
\[
	u = e^{- i \s t_*}v_\s,
\]
where
\[
	\left(\d_{t_*} + \frac a{r_0^2 + a^2} \d_{\phi_*}\right) v_\s = 0.
\]
\end{remark}

Our first main result is the following:

\begin{thm}[Discrete set of quasinormal modes] \label{thm: QNMs} \ \\
Let $(M_*, g_*)$ and $P$ be as in Assumption \ref{ass: main}.
Then there is a discrete set of quasinormal mode frequencies.
More precisely, there is a discrete set $\A \subset \C$ such that
\[
	\s \in \A
\]
if and only if there is a quasinormal mode 
\[
	u \in C^\infty(M_*)
\] 
with frequency $\s$.
Moreover, for each $\s \in \A$, the space of quasinormal modes is finite dimensional.
If the coefficients of $P$ are real analytic, then the quasinormal modes are real analytic.
\end{thm}

\noindent
The discrete set of quasinormal modes on the Kerr-de Sitter spacetime is analogous to eigenfunctions of an elliptic operator on e.g.\ a compact manifold without boundary.
The quasinormal mode frequencies are analogous to the corresponding eigenvalues.
Just as eigenvalues and eigenfunctions depend on the operator, i.e.\
generally change when the coefficients change, the quasinormal modes
and frequencies of course also depend on the operator.

This theorem is proved in Section~\ref{sec:Fredholm} by showing that
$\s\in\A$ is equivalent to the lack of invertibility of a certain
Fredholm operator $P_\s$. 
Since $P_\s$ depends analytically on $\s$ in the appropriate
sense, the analytic Fredholm theorem guarantees the discreteness of
$\A$.

A difference between the eigenfunctions and eigenvalues of elliptic
self-adjoint operators and the present case is that the meromorphic
family $P_\s^{-1}$ can have higher order poles, though they are
finite rank. These Laurent coefficients then give rise to {\em
  generalized quasinormal modes} which play a role in the asymptotic
expansion of solutions of wave equations below.

While the Killing vector field in the definition of quasinormal modes
may seem curious, and thus the condition on $v_\s$ in
Remark~\ref{rem:quasi} odd, one way to think about this, and indeed
this is how the proof proceeds in Section~\ref{sec:Fredholm}, is to change coordinates, namely
replace $\phi_*$ by $\psi_*=\phi_*-\frac{a}{r_0^2+a^2}t_*$, while
keeping $t_*$ unchanged (but call it $\tau_*=t_*$ for clarity). In the
new coordinates, the quasinormal modes are functions of the new
`spatial' variables, $r,\theta,\psi_*$, i.e.\ are annihilated by $\partial_{\tau_*}$. One can think of this step as
a refinement of the earlier coordinate changes (namely the introduction of
$t_*,\phi_*$ in place of $t,\phi$) that were necessitated by the horizons. While the metric
is already well-behaved after those coordinate changes, thus the wave
operator has smooth coefficients and is non-degenerate, the solvability
theory for non-elliptic operators is much more intricate than for
elliptic operators, and this further refinement plays a key role in
the analysis.

\subsubsection{Asymptotic expansion}

Our second main result concerns the behavior of solutions to the
wave equation. 
For the statement of this, we use the standard Sobolev spaces based on the cylindrical geometry of $M_*$, i.e.\ use the vector fields $\partial_{t_*}$, $\partial_r$ and vector fields on the sphere. 
For instance, considering $(r_e-\delta,r_c+\delta)_r \times S^2_{\theta,\phi_*}\subset\R^3_y$ as spherical coordinates, for non-negative integers $s$ we simply have
$$
	\|u\|_{\bar H^s}^2=\sum_{j+|\beta|\leq s}\|\partial_{t_*}^j\partial_y^\beta u\|^2_{L^2(M_*,dg_*)}.
$$
Here the bar over $H$ corresponds to H\"ormander's notation for
extendible distributions, see \cite{HorIII}. 
We remark that if one compactifies
$M_*$ via replacing $t_*$ by $T_*=e^{-t_*}$, adding $T_*=0$ as an
ideal boundary, then these spaces are Melrose's b-Sobolev spaces, see
\cite{Me1993}, which is how the result was phrased in \cite{V2013}. Changing to the variables $(\tau_*,\psi_*)$ leaves this
definition unchanged, up to equivalence of norms, i.e.\ we could
equally well write, with $(r_e-\delta,r_c+\delta)_r \times S^2_{\theta,\psi_*}\subset\R^3_z$ being spherical coordinates,
$$
\|u\|_{\bar H^s}^2=\sum_{j+|\beta|\leq s}\|\partial_{\tau_*}^j\partial_z^\beta u\|^2_{L^2(M_*,dg_*)}.
$$

We then have the following:

\begin{thm}[The asymptotic expansion of waves] \label{thm: asymptotic expansion} \ \\
Let $(M_*, g_*)$ and $P$ be as in Assumption \ref{ass: main} and let $t_0 \in \R$.
There are $C, \de > 0$ such that for $0 < \e < C$ and
\[
	s > \frac12 + \b \e,
\]
where $\b$ is defined in \eqref{eq: beta}, any solution to
\[
	Pu = f
\]
with $f \in e^{- \e t_*} \bar H^{s - 1 + \de}(M_*)$ and with $\supp(u) \cup \supp(f) \subset \{t_* > t_0\}$ has an asymptotic expansion
\begin{align*}
	u - \sum_{j = 1}^N \sum_{k = 0}^{k_j} t_*^k e^{- i \s_j t_*} v_{jk} \in e^{- \e t_*} \bar H^s(M_*),
\end{align*}
where $\s_1, \hdots, \s_N$ are the (finitely many) quasinormal mode frequencies with 
\[
	\Im \s_j > - \e
\]
and $k_j$ is their multiplicity, and where $e^{-i\s_j t_*}v_{jk}$ are
the $C^\infty$ (generalized) quasinormal modes with frequency $\s_j$ which are
real analytic if the coefficients of $P$ are such.
\end{thm}

Theorem \ref{thm: asymptotic expansion} implies that we always have decay apart from finitely many terms.
Moreover, decay for all solutions $u$ is equivalent to proving that
\[
	\Im(\s) < 0,
\]
for all $\s \in \A$. 
On the other hand, decay for all solutions to the zero resonance amounts to proving that for all $\s\in\A\setminus\{0\}$, $\Im(\s)<0$, and $\s=0$ is simple (with thus no generalized quasinormal modes).
This is what is commonly known as \emph{mode stability} for the operator $P$, which in general will depend strongly on the lower order terms. 
For this, only the relatively simple case of $|a|$ small is
known, see Dyatlov's paper \cite{D2011}: for the wave operator, if
$a=0$, this (in the second sense, corresponding to decay to constants) can be shown explicitly, and if $|a|$ is small,
it follows by perturbation stability of the Fredholm setup; for
the Klein-Gordon operator of small positive mass parameter the first (full
decay) statement holds for small $|a|$. We remark
that recently mode stability was extended to a larger range of
parameters by Casals and Texeira da Costa \cite{CTC2021}.

\subsubsection{Quasilinear wave equations}
In fact, our results immediately make the methods of \cite{HV2016} on
solving quasilinear equations
applicable in the extended range of parameters. 
That paper uses the compactification $\overline{M_*}$ of $M_*$ by
adding $e^{-t_*}$ as a boundary defining function, mentioned above, to
define the b-Sobolev spaces (which we here write as $\bar H$,
following the discussion above), and b-differential operators
$\Diffb$; here we merely state a simplified result as an illustration. Using $(t_*,y)$ as above, $\md u=(\partial_{t_*}u, \md_y u)$, for each $p=(t_*,y)$ we have an inner product $g_p(u(p),\md u(p))$ on $T_pM$, where $g_p:\R\oplus T_p^*M\to S^2T_p^*M$ depends smoothly, up to the
boundaries $r=r_e-\delta$, $r=r_c+\delta$, on $p$ via $y$ (or indeed via $(z,e^{-\tau_*})$), and we consider the quasilinear wave equation
\[
	\Box_{g(u,\md u)}u=f+q(u,\md u)
\]
for real-valued $u$,
with $q$ being a polynomial in $u,\partial_{t_*}u,\partial_y u$ with vanishing
zeroth and first order terms, and with real coefficients that are smooth
functions of $y$ (or again $(z,e^{-\tau_*})$).
The weighted version of the Sobolev spaces is simply $\bar H^{s,\alpha}$ given by $u\in\bar H^{s,\alpha}$ if $e^{\alpha t_*}u\in \bar H^s$.
For instance, Theorem~4 of \cite{HV2016}, in a simplified form
(corresponding to Theorems~1 and 2 there) but with an extended range of parameters becomes:

\begin{thm}[cf.\ Theorems~4 and 1 of \cite{HV2016}] \ \\
Let $(M_*,g_*)$ be a subextremal Kerr-de Sitter spacetime, and consider
$L_{g(u,d u)}=\Box_{g(u,\md u)}$ with $g(0,0)=g_*$. 
Suppose further that $L_0=L_{g(0,0)}$ is such that $L_0$ has a simple
resonance at $0$ (i.e.\ the associated Fredholm operator's inverse has
a simple pole), with resonant states spanned by constants, and no other resonances in $\Im \s\geq 0$, i.e.\ that mode stability in the second sense holds.
Suppose that $q(u_0, \md u_0)=0$ for $u_0$ a constant.
Let $s\in\R$.

Then, with $d=12$, for $\alpha>0$ sufficiently small: 
If $f\in\bar H^{\infty,\alpha}$ is real valued with a sufficiently
small $\bar H ^{2d,\alpha}$-norm (depending on $s$), then the equation
$L_{g(u, \md u)} u=f+q(u, \md u)$ has a unique, smooth in $M_*$, real
valued, global forward solution of the form $u=u_0+\tilde u$, $\tilde
u\in\bar H^{s,\alpha}$, $u_0=c\chi$, $c$ a constant, $\chi\in
C^\infty(M_*)$ identically $1$ for $t_*$ large.
\end{thm}

\section{The Fredholm setup}\label{sec:Fredholm}

The goal of this section is to prove Theorem \ref{thm: QNMs}.
This generalizes \cite{V2013} by \emph{removing} the assumption
\begin{equation} \label{eq: old condition 1}
    \left(1 - \frac{\Lambda a^2}3\right)^3 > 9 \Lambda m^2,
\end{equation}
which was required in \cite{V2013}.
In fact, \eqref{eq: old condition 1} was stated in \cite{V2013}*{(6.13)} as the equivalent condition
\[
	r_0 \in (r_e, r_c), \quad \mu'(r_0) = 0 \quad \Rightarrow \quad a^2 < \mu(r_0),
\]
i.e.\ that the maximum point of $\mu$ in the domain of outer communication is larger than $a^2$.
This will not be necessary in the analysis we present below.
Vasy considered in \cite{V2013} the operator
\[
    \hat P_\s u := e^{i \s t_*} P\left(e^{- i \s t_*} u \right),
\]
for $\s \in \C$, where $t_*$ is as above (or a slight modification with similar properties), and $u$ only depends on the remaining coordinates $(\phi_*, r, \theta)$.
This corresponds to the condition
\[
	\d_{t_*} u = 0.
\]
One may therefore consider $\hat P_\s$ as a linear second order differential operator
\[
	\hat P_\s : C^\infty(L_*) \to C^\infty(L_*),
\]
where
\[
	L_* := \{t_* = 0\} \times (r_e - \de, r_c + \de)_r \times S^2_{\psi_*, \theta} \subset M_*.
\]
In \cite{V2013} it was shown that $\hat P_\s$ is a Fredholm operator between appropriate Sobolev spaces, \emph{assuming} \eqref{eq: old condition 1}.
If \eqref{eq: old condition 1} is violated, the Fredholm theory of \cite{V2013} does not apply to $\hat P_\s$.
The main reason for this is that there are trapped bicharacteristics of $\hat P_\s$ in $T^*(L_* \cap M)$ if $a$ is too large.
This is closely related to the fact that the ergoregions of the event horizon and cosmological horizon intersect for large $a$.
In order to avoid this, we will construct another operator $P_\s$, which has similar properties in the full subextremal range as $\hat P_\s$ has for $a$ satisfying \eqref{eq: old condition 1}.
For this, we first introduce a new coordinate system $(\tau_*, r, \psi_*, \theta)$, where
\begin{equation} \label{eq: star coordinates}
	\begin{pmatrix}
	    \tau_* \\
		\psi_*
	\end{pmatrix} 
		:= 
	\begin{pmatrix}
	    t_* \\
		\phi_* - \frac a{r_0^2 + a^2} t_*
	\end{pmatrix},
\end{equation}
with $r_0 \in (r_e, r_c)$ being uniquely defined by
\[
    \mu'(r_0) = 0.
\]
Note that
\[
    \d_{\tau_*}
        = \d_{t_*} + \frac a {r_0^2 + a^2} \d_{\phi_*}, \quad
    \d_{\psi_*}
        = \d_{\phi_*}
\]
are both Killing vector fields, since $a$ and $r_0$ are constant.
A better choice than $\hat P_\s$ is to consider the operator
\[
    P_\s u 
    	:= e^{i \s \tau_*} P\left(e^{- i \s \tau_*} u \right)
    	= e^{i \s t_*} P\left(e^{- i \s t_*} u \right),
\]
where $u$ only depends on coordinates $(\psi_*, r, \theta)$, i.e.
\begin{equation} \label{eq: partial tau star}
    \d_{\tau_*} u = \left( \d_{t_*} + \frac a {r_0^2 + a^2} \d_{\phi_*} \right) u = 0.
\end{equation}
Even though $\tau_* = t_*$, we get a quite different induced operator $P_\s$ on the modes, assuming $\d_{\tau_*} u = 0$ instead of $\d_{t_*}u = 0$.
We conclude that
\[
    P_\s: C^\infty(L_*) \to C^\infty(L_*)
\]
is a linear differential operator of second order, where we note that indeed
\[
	L_* = \{{\tau_*} = 0\} \times (r_e - \de, r_c + \de)_r \times S^2_{\psi_*, \theta}.
\]
The key difference to $\hat P_\s$ is that the there are no trapped bicharacteristics of $P_\s$ in $T^*(L_* \cap M)$ for any parameters in the subextremal range, see Lemma \ref{le: bicharacteristics} below.

This will be used to prove that $P_\s$ is a Fredholm operator between appropriate spaces, which is the main step in proving Theorem \ref{thm: QNMs}.
In order to formulate the Fredholm statement, define $\b_e, \b_c \in \R$ as
\[
	\b_{e/c} := \pm 2 \left(1 + \frac{\Lambda a^2}3\right) \frac{r^2 + a^2}{\mu'}|_{r = r_{e/c}}.
\]
Since $\mu'(r_e) > 0$ and $\mu'(r_c) < 0$, we note that $\b_{e/c} > 0$.
For each $s \in \R$, we use the notation
\[
    \bar H^s := \bar H^s\left(L_*\right),
\]
for the extendible Sobolev distributions on $L_*$, in the sense of H\"ormander \cite{HorIII}.
We will prove the following modification of \cite{V2013}*{Theorem~1.1}, which holds in the full subextremal range:
\begin{thm} \label{thm: Fredholm}
Define
\begin{equation}\label{eq: beta}
	\b := \max(\b_e, \b_c) > 0
\end{equation}
and let $s \geq \frac12$.
The operator
\[
    P_\s: \{u \in \bar H^s \mid P_\s u \in \bar H^{s-1} \} \to \bar H^{s-1}
\]
is an analytic family of Fredholm operator of index $0$ for all $\s \in \C$ such that
\[
	\Im \s > \frac{1 - 2s}{2\b}.
\]
Moreover, $P_\s$ is invertible for $\Im \s \gg 1$.
\end{thm}

The proof of this theorem follows from the Fredholm framework for non-elliptic operators, developed in \cite{V2013}, once we have established the necessary results for the bicharacteristics, described above.
We begin by studying the bicharacteristics in the domain of outer communication $M$ and it will be convenient to do the computations in a modification of the Boyer-Lindquist coordinates $(t, r, \phi, \theta)$.
We define these new coordinates analogous to \eqref{eq: star coordinates} as
\[
	\begin{pmatrix}
		\tau \\
		\psi
	\end{pmatrix} 
		:= 
	\begin{pmatrix}
	    t \\
		\phi - \frac a{r_0^2 + a^2} t
	\end{pmatrix}.
\]
Since
\[
    \d_\tau = \d_t + \frac a{r_0^2 + a^2} \d_\phi = \left(\d_{t_*} + \frac a{r_0^2 + a^2} \d_{\phi_*}\right)|_M = \d_{\tau_*}|_M,
\]
we may identify the restriction of $P_\s$ to $L_* \cap M$ with the operator
\[
	e^{- i \s \tau}P \left( e^{i \s \tau} u \right),
\]
where $u \in C^\infty(M)$, such that
\[
    \d_\tau u = 0.
\]
We may therefore identify $P_\s|_{L_* \cap M}$ with the linear the second order differential operator 
\[
	P_\s : C^\infty(L) \to C^\infty(L),
\]
where 
\[
    L := \{\tau = 0\} \times (r_e, r_c)_r \times S^2_{\psi, \theta}.
\]
In the new coordinates, the dual metric $G$ is given by
\begin{align*}
	(r^2 + a^2\cos^2(\theta)) G
		&= \mu(r) \d_r^2 + c(\theta) \d_\theta^2 \\*
		&\qquad + \frac{b^2}{c(\theta)\sin^2(\theta)} \left(a \sin^2(\theta) \d_\tau + \frac{r_0^2 + a^2 \cos^2(\theta)}{r_0^2 + a^2} \d_\psi\right)^2 \\*
		&\qquad - \frac{b^2}{\mu(r)} \left( (r^2 + a^2) \d_\tau + a \frac{r_0^2 - r^2}{r_0^2 + a^2} \d_\psi \right)^2.
\end{align*}
The principal symbol $\p_\s$ of $P_\s$ is thus given by
\begin{align}
	&(r^2 + a^2\cos^2(\theta))\p_\s(\xi) \nonumber \\
		&\quad =  \mu(r) \xi_r^2 + c(\theta) \xi_\theta^2 + \frac{b^2}{(r_0^2 + a^2)^2} \left( \frac{(r_0^2 + a^2 \cos^2(\theta))^2}{c(\theta) \sin^2(\theta)} - a^2 \frac{(r_0^2 - r^2)^2}{\mu(r)} \right) \xi_\psi^2. \label{eq: principal symbol classical}
\end{align}
Since the bicharacteristic flow is invariant under conformal rescaling, it suffices to consider
\[
   \q_\s
        := (r^2 + a^2\cos^2(\theta))\p_\s.
\]
\begin{lemma}[No characteristic set at $r_0$] \label{le: no bich r 0}
We have
\[
	\Char(P_\s) \subset \{ r \neq r_0 \}.
\]
\end{lemma}
\begin{proof}
Assume that there is a point in the characteristic set with $r = r_0$.
Then
\begin{align*}
	0 
	    = \q_\s(\xi)|_{r = r_0}
		= \mu(r_0) \xi_r^2 + c(\theta) \xi_\theta^2 + \frac{b^2}{(r_0^2 + a^2)^2} \frac{(r_0^2 + a^2 \cos^2(\theta))^2}{c(\theta) \sin^2(\theta)} \xi_\psi^2,
\end{align*}
and since $\mu(r_0) > 0$, this is a sum of positive terms and hence
\[
	\xi_r = \xi_\theta = \xi_\psi = 0.
\]
This proves the lemma.
\end{proof}

\begin{remark}[Ergoregions]
One way of interpreting the above lemma is that $\d_\tau$ is \emph{timelike} at $r_0$. 
Hence the ergoregions of the event horizon and cosmological horizon, with respect to $\d_\tau$, are disjoint.
It is easy to see that this is not the case for $\d_t$ in general; the ergoregions may well intersect in that case.
\end{remark}

\begin{lemma}[No trapping] \label{le: bicharacteristics}
For each $\e > 0$, all bicharacteristcs of $P_\s$ in $L$ leave the region
\[
	(r_e + \e, r_c - \e)_r \times S^2_{\psi, \theta}
\]
both to the future and past.
\end{lemma}

\begin{proof}
The Hamiltonian vector field is given by
\[
	\H_{\q_\s} = \sum_{j = 1}^3 (\d_{\xi_j}\q_\s)\d_j - (\d_j \q_\s)\d_{\xi_j}.
\]
We claim that 
\begin{equation}\label{eq: classical Hamilton implication}
	\H_{\q_\s} r = 0 \quad \Rightarrow \quad 	\H_{\q_\s}^2 r \
	\begin{cases}
		< 0, \quad \text{ if } r_e < r < r_0, \\
		> 0, \quad \text{ if } r_0 < r < r_c,
	\end{cases}
\end{equation}
in the characteristic set.
Assume therefore that
\begin{align*}
	0
		&= \H_{\q_\s} r = \d_{\xi_r} \q_\s = 2 \xi_r \mu(r).
\end{align*}
Since $\mu(r) > 0$ for all $r \in (r_e, r_c)$, we conclude that $\xi_r = 0$.
At such points, the second derivative is given by
\begin{align*}
	\H_{\q_\s}^2 r|_{\xi_r = 0}
		&= 2 \left( \H_{q_\s} \xi_r|_{\xi_r = 0} \right) \mu(r) \\
		&= - 2 \left( \d_r \q_\s|_{\xi_r = 0} \right) \mu(r) \\
		&= 2 \frac{b^2}{(r_0^2 + a^2)^2}  a^2 \d_r \frac{(r_0^2 - r^2)^2}{\mu(r)} \xi_\psi^2 \\
		&= 2 \frac{a^2 b^2}{(r_0^2 + a^2)^2} \frac{\mu(r) \d_r (r_0^2 - r^2)^2 - \mu'(r)(r_0^2 - r^2)^2}{\mu(r)^2} \xi_\psi^2.
\end{align*}
Now, with our choice of $r_0$, we note that 
\[
	\mu'(r) 
	   \begin{cases} 
	       > 0, \quad r_e < r < r_0, \\ 
	       < 0, \quad r_0 < r < r_c,
	   \end{cases}
\]
and
\[
	\d_r (r_0^2 - r^2)^2 
		\begin{cases} 
			< 0, \quad r_e < r < r_0, \\
			> 0, \quad r_0 < r < r_c.
		\end{cases}
\]
Finally, $(r_0^2 - r^2)^2 > 0$ for $r \in (r_e, r_c) \backslash \{r_0\}$ and $\mu(r) > 0$ for all $r \in (r_e, r_c)$, proving \eqref{eq: classical Hamilton implication}.
We may use this to define an escape function
\[
	\E := e^{C(r - r_0)^2}\H_{\q_\s} r
\]
for any $C > 0$ and note that
\[
	\H_{\q_\s} \E
		= e^{C(r - r_0)^2} \left( 2 C(r - r_0) \left(\H_{\q_{\s}}r\right)^2 +  \H_{\q_\s}^2 r \right).
\]
Since the characteristic set is disjoint from $\{r = r_0\}$ and by \eqref{eq: classical Hamilton implication}, we can choose $C$ large enough so that $\H_{\q_\s} \E$ is everywhere non-vanishing and has the same sign as $r - r_0$.
Hence $\E$ gives an escape function for the bicharacteristics of $\q_\s$ in $L$, which finishes the proof.
\end{proof}
We now improve on Lemma \ref{le: bicharacteristics} by including the precise behavior at and beyond the horizons.
By Assumption \ref{ass: main}, the hypersurface $L_*$ is spacelike, which means that $\md t_*$ is a future pointing timelike one-form everywhere along $L_*$.
All elements of the characteristic set of $P$ are lightlike, so we can use this fact to divide the characteristic set of $P$ into two disjoint sets.
Moreover, since there is a natural embedding
\[
	\Char(P_\s) \subseteq \Char(P),
\]
we may divide the characteristic set of $P_\s$ as
\[
	\Char(P_\s)
		= \S_+ \cup \S_-, 
\]
where 
\[
	\S_\pm
		= \{ \xi \in \Char(P_\s) \mid \pm G_*(\md t_*, \xi) > 0 \},
\]
where $G_*$ is the dual metric induced by $g_*$, defined in \eqref{eq: metric at horizons}.
Note that indeed
\[
	\S_+ \cap \S_-
		= \emptyset,
\]
hence $\S_+$ and $\S_-$ are invariant under the bicharacteristic flow.
\begin{lemma} \label{le: classical radial points} 
Each bicharacteristic of $P_\s$ in $\S_+$ either starts at fiber infinity of
\[
	N^*\{r = r_e\} \cap \{\xi_r > 0\}
\]
and ends at $r = r_e - \de$ or starts at the fiber infinity of
\[
	N^*\{r = r_c\} \cap \{\xi_r < 0\}
\]
and ends at $r = r_c + \de$.
The reverse behavior holds for $\S_-$.
Moreover, the fiber infinity of
\[
	N^*\{r = r_e\} \cap \{\pm \xi_r > 0\} \text{ and } N^*\{r = r_c\} \cap \{\mp \xi_r > 0\}
\]
are generalized normal source/sink manifolds of the bicharacteristic flow, respectively, in the sense of \cite{V2013}.
\end{lemma}

\begin{proof}
This actually follows from the semiclassical considerations in \cite{V2013}, since in a large parameter sense (the relationship between large parameter and semiclassical frameworks being discussed in Section~2.1 there) our new choice of $\psi_*$ amounts to a new choice of the `classical hyperplane' $\sigma=0$ relative to the choice in \cite{V2013} (the classical hyperplane in \cite{V2013} corresponds to the choice $\sigma+\frac{a^2}{r_0^2+a^2}\xi_{\psi_*} = 0$ in the notation of this paper), but at the normal source/sink manifolds, $\xi_{\psi_*}=0$, so these lie at the same location, and the normal source/sink manifold being such even in the large parameter sense for the old choice (which is equivalent to the semiclassical considerations there) implies the same property for our new choice. 
However, for the convenience of the reader we give a direct argument. 

In order to study the characteristic set of $P_\s$ near the horizons, we need to work in the coordinate system $(\tau_*, r, \psi_*, \theta)$, defined in \eqref{eq: star coordinates}.
The dual metric in these coordinates is given by
\begin{align*}
	&(r^2 + a^2 \cos^2(\theta))G_* \\
		&\quad = \mu(r) \d_r^2 - 2 b f(r) \left( (r^2 + a^2) \d_{\tau_*} + a \frac{r_0^2 - r^2}{r_0^2 + a^2} \d_{\psi_*} \right) \d_r \\*
		&\quad \qquad + c(\theta) \d_\theta^2 + \frac{b^2}{c(\theta)\sin^2(\theta)}\left(a \sin^2(\theta)\d_{\tau_*} + \frac{r_0^2 + a^2 \cos^2(\theta)}{r_0^2 + a^2} \d_{\psi_*}\right)^2 \\*
		&\quad \qquad - b^2 \frac{1 - f(r)^2}{\mu(r)} \left( (r^2 + a^2) \d_{\tau_*} + a \frac{r_0^2 - r^2}{r_0^2 + a^2} \d_{\psi_*} \right)^2.
\end{align*}
The principal symbol $\p_\s$ of $P_\s$ is thus given by
\begin{align}\label{eq:principal-symb-classical}
	&(r^2 + a^2\cos^2(\theta))\p_\s(\xi) \nonumber \\
		&\quad = \mu(r) \xi_r^2 - 2 a b f(r) \frac{r_0^2 - r^2}{r_0^2 + a^2} \xi_{\psi_*} \xi_r + c(\theta) \xi_\theta^2 \nonumber \\*
		&\quad \qquad + \Bigg( \frac{b^2}{c(\theta)\sin^2(\theta)}\left(\frac{r_0^2 + a^2 \cos^2(\theta)}{r_0^2 + a^2} \right)^2 - b^2 \frac{1 - f(r)^2}{\mu(r)} \left(a \frac{r_0^2 - r^2}{r_0^2 + a^2} \right)^2 \Bigg) \xi_{\psi_*}^2. 
\end{align}
As before, since the bicharacteristic flow is invariant under conformal rescaling, it suffices to consider
\[
   \q_\s
        := (r^2 + a^2\cos^2(\theta))\p_\s.
\]
By Lemma \ref{le: no bich r 0}, there are no bicharacteristics crossing $r = r_0$.
In this proof, we therefore only study the bicharacteristics where $r < r_0$, as the case $r > r_0$ is similar.
Since
\[
	G_*(\md r, \md r)|_{r = r_e} 
		= 0 
\]
and
\[
	\md t_* = \md \tau_*,
\]
we have
\[
	 \quad G_*(\md t_*, \md r)|_{r = r_e} 
	 	= G_*(\md \tau_*, \md r)|_{r = r_e} 
	 	= b (r_e^2 + a^2) 
	 	> 0.
\]
This means in particular that $- \md r|_{r = r_e}$ is a future pointing lightlike one-form, so it can be used to characterize $\S_\pm \cap \{r = r_e\}$.
For any $\xi \in \Char(P_\s) \cap \{r = r_e\}$, it follows that $\xi \in \S_\pm \cap \{r = r_e\}$ if and only if
\[
	\pm G_*(\md t_*, \xi)|_{r = r_e} > 0,
\]
which is equivalent to
\[
	\pm G_*(-\md r, \xi)|_{r = r_e} > 0
\]
or 
\[
	\xi|_{r = r_e} = \xi_r \md r|_{r = r_e},
\]
with $ \pm \xi_r|_{r = r_e} > 0$.
We now compute the Hamiltonian vector field applied to the function $r$ at the event horizon:
\begin{equation} \label{eq: Ham r event}
\begin{split}
	\H_{\q_\s}  r|_{r = r_e}
		&= \sum_{j = 1}^3 \left( (\d_{\xi_j}\q_\s)\d_j - (\d_j \q_\s)\d_{\xi_j} \right) r |_{r = r_e} \\
		&= 2 ab \frac{r_0^2 - r_e^2}{r_0^2 + a^2} \xi_{\psi_*}|_{r = r_e} \\
		&= - G(-\md r, \xi)|_{r = r_e}.
\end{split}
\end{equation}
We conclude that all bicharacteristics in $\S_+$ crossing the event horizon enter the black hole region, where $r < r_e$, and all bicharacteristics in $\S_-$ crossing the event horizon enter the domain of outer communication, where $r \in (r_e, r_c)$.
In particular, no bicharacteristic can pass the event horizon twice.
On the other hand, Lemma \ref{le: bicharacteristics} implies that each bicharacteristic in $\S_{\pm}$ in the domain of outer communication  turns in the past/future and approaches the event horizon $r = r_e$. \\
It remains to show that the fiber infinity of 
\[
	N^*\{r = r_e\} \cap \{ \pm \xi_r > 0\}
\]
is a generalized normal source/sink of the bicharacteristic flow, in the sense of \cite{V2013}.
Let us first check that $N^*\{r = r_e\}$ is the largest subset of
\[
	\Char(P_\s) \cap \{r =r_e\},
\]
which is invariant under the bicharacteristic flow.
Indeed, equation \eqref{eq: Ham r event} shows that the Hamiltonian vector field is transversal unless $\xi_{\psi_*} = 0$.
But if $\xi_{\psi_*} = 0$, then equation \eqref{eq:principal-symb-classical} implies that $\xi_\theta = 0$ and hence
\[
	\xi \in N^*\{r = r_e\}.
\]
Restricting the Hamiltonian vector field to the conormal bundle, we get
\[
	\xi_r^{-1} \H_{\q_\s}|_{N^*\{r = r_e\}} 
		= 2 a b \frac{r_0^2 - r_e^2}{r_0^2 + a^2} \d_{\psi_*} - \mu'(r_e) \xi_r \d_{\xi_r}.
\]
Since $\mu'(r_e)>0$, this shows that the fiber infinity of 
\[
	N^*\{r = r_e\} \cap \{ \pm \xi_r > 0\}
\]
is a generalized source/sink of the bicharacteristic flow, in the sense of \cite{V2013}.
We finally need to check that it indeed is a \emph{normal} source/sink in the generalized sense of \cite{V2013}. 
This is to say that for a suitable (local) quadratic defining function $\rho_0$ for $N^*\{r = r_e\}$ in the characteristic set of $P_\s$, modulo cubicly vanishing terms at $N^*\{r = r_e\}$, there is a function $\beta_1 > 0$ such that
\[
	\xi_r^{-1}\H_{\q_\s}\rho_0 - \beta_1 \rho_0
		\geq 0
\]
near $N^*\{r = r_e\}$.
We can for example take $\rho_0$ to be a multiple of the (modified) Carter constant, i.e.
\[
	\rho_0
		:= \xi_r^{-2} \left( c(\theta) \xi_\theta^2 + \frac{b^2}{c(\theta)\sin^2(\theta)}\left(\frac{r_0^2 + a^2 \cos^2(\theta)}{r_0^2 + a^2} \right)^2 \xi_{\psi_*}^2 \right).
\]
Note that
\begin{align*}
	\H_{\q_\s} \xi_r^2 \rho_0
		&= \{\q_\s, \xi_r^2 \rho_0 \} \\
		&= \Bigg\{ \mu(r) \xi_r^2 - 2 a b f(r) \frac{r_0^2 - r^2}{r_0^2 + a^2} \xi_{\psi_*} \xi_r - b^2 \frac{1 - f(r)^2}{\mu(r)} \left(a \frac{r_0^2 - r^2}{r_0^2 + a^2} \right)^2 \xi_{\psi_*}^2, \\
		&\quad \qquad c(\theta) \xi_\theta^2 + \frac{b^2}{c(\theta)\sin^2(\theta)}\left(\frac{r_0^2 + a^2 \cos^2(\theta)}{r_0^2 + a^2} \right)^2 \xi_{\psi_*}^2 \Bigg\} + \{\xi_r^2 \rho_0, \xi_r^2 \rho_0\} \\
		&= 0,
\end{align*}
since the second Poisson bracket vanishes trivially and the first Poisson bracket vanishes because the first factor only depends on $(r, \xi_r, \xi_{\psi_*})$ and the second factor only depends on $(\theta, \xi_\theta, \xi_{\psi*})$.
We deduce that
\begin{align*}
  \xi_r^{-1}\H_{\q_\s}\rho_0
  	= \xi_r\rho_0 \H_{\q_\s}\xi_r^{-2}
  	= - 2 \xi_r^{-2} (\H_{\q_\s}\xi_r)
  \rho_0,
\end{align*}
so $\beta_1 = - 2 \xi_r^{-2} (\H_{\q_\s}\xi_r)$.
At $N^*\{r = r_e\}$, we have
\[
	\beta_1
		= 2 \mu'(r_e),
\]
which is positive as desired. 
We have thus shown that the fiber infinity in 
\[
	N^*\{r = r_e\} \cap \{\pm \xi_r > 0\}
\]
is a normal source/sink in the generalized sense of \cite{V2013}.
This concludes the proof, since the behavior near $r = r_c$ is studied analogously.
\end{proof}

\begin{proof}[Proof of Theorem \ref{thm: Fredholm}]
By Lemma \ref{le: classical radial points}, the dynamics of the bicharacteristics at and beyond the horizons $\Ho_e^+ \cap L_*$ and $\Ho_c^+ \cap L_*$ is precisely the same as in \cite{V2013}*{Section~6.1}, i.e.\ the analysis for $P_\s$ in our setup is similar to the analysis for $\hat P_\s$ there.
The proof of Theorem \ref{thm: Fredholm} therefore follows the same lines as the proof of \cite{V2013}*{Thm.\ 1.4}.
\end{proof}

\begin{proof}[Proof of Theorem \ref{thm: QNMs}]
Since $P_\s$ is invertible for $\Im \s \gg 1$, analytic Fredholm theory implies that $P_\s$ has a meromorphic extension to the open set
\[
	\Omega_s := \left\{ \Im \s > \frac{1 - 2s}{2\b} \right\}.
\]
In particular, $P_\s$ is invertible everywhere in $\Omega_s$ apart form a discrete set.
Moreover, since $P_\s$ has index zero, $P_\s$ is invertible if and only if the kernel of $P_\s$ is trivial.
Since 
\[
	\C = \bigcup_{s \in \R} \Omega_s,
\]
we conclude that $\ker(P_\s)$ is non-trivial precisely on a discrete set $\A \subset \C$.
Following the arguments in the proof of \cite{PV2021}*{Theorem~1.2} line by line, using Theorem \ref{thm: Fredholm} in place of the \cite{V2013}*{Theorem~1.1}, it follows that smooth elements in $\ker(P_\s)$ are real analytic if the coefficients of $P$ are real analytic.
\end{proof}

\begin{remark}
We recall that \cite{PV2021} proves the real analyticity of the
quasinormal modes by using yet
another Killing vector field, or rather two, one each for the two horizons,
with respect to which the Fourier transformed wave operator is of
Keldysh type, so has a similar structure to that of the
Schwarzschild-de Sitter spacetime relative to the standard
$\partial_{t_*}$. (The Killing vector fields we use are lightlike at the
horizon under study.) Hence, after this reduction, the real analyticity
result of Galkowski and Zworski \cite{GZ2020} can be used in a local
manner at each horizon. Then the
mode analyticity with respect to any other Killing vector field which
gives rise to a global Fredholm theory, such as ours presently, is
deduced by decomposing the quasinormal modes into eigenmodes relative
to the vector field $\partial_{\psi_*}$, each of which is a quasinormal
mode relative to the horizon Killing vector fields as well.
  \end{remark}

\section{Normally hyperbolic trapping}

The goal of this section is to prove Theorem \ref{thm: asymptotic expansion}.
This generalizes \cite{V2013} by \emph{removing} the assumption
\begin{equation} \label{eq: old condition 2}
    \abs{a} \leq \frac{\sqrt 3}{2}m,
\end{equation}
which was required in \cite{V2013}*{(6.27)}.
In the previous section, we studied the mode operator and showed in particular that the bicharacteristics of that operator are non-trapped.
As is well-known, there are trapped bicharacteristics for the full wave operator, i.e.~trapped lightlike geodesics.

In order to apply the \emph{semi-classical} or \emph{high energy} estimates from \cite{V2013}, we need to prove certain properties of the trapping in the domain of outer communication.
More precisely, we need to show that the trapping is \emph{normally hyperbolic}.
This was done in \cite{V2013} assuming \eqref{eq: old condition 2}.
In this section, we prove that the analogous results hold for the full subextremal range.
Since we only work in the domain of outer communication, it is convenient to work in Boyer-Lindquist coordinates $(t, r, \phi, \theta)$ with dual variables $(\xi_t, \xi_r, \xi_\phi, \xi_\theta)$. 
We let $\p$ denote the principal symbol of the wave operator $P$.

\begin{remark}
Since $\d_t$ and $\d_\phi$ are Killing vector fields, it follows that $\xi_t$ and $\xi_\phi$ are constant along the Hamiltonian flow with respect to $\p$.
\end{remark}

\begin{thm}[Trapping in the subextremal Kerr-de Sitter spacetimes] \label{thm: trapping in Kerr-de Sitter} 
For any $(\xi_t, \xi_\phi) \in \R^2 \backslash \{(0, 0)\}$, define the function
\[
    F(r) := \frac1\mu \left((r^2 + a^2)\xi_t + a \xi_\phi \right)^2.
\]
\begin{enumerate}[label=(\alph*)]
\item \label{claim: uniqueness} Either 
\begin{itemize}
	\item $F$ vanishes at $r = r_e$ or $r_c$ and $F$ and has no critical point in $(r_e, r_c)$,
\end{itemize}
or
\begin{itemize} 
	\item $F$ has precisely one critical point $r_{\xi_t, \xi_\phi} \in (r_e, r_c)$ and $F''\left(r_{\xi_t, \xi_\phi}\right) > 0$.
\end{itemize}
\item \label{claim: positivity} $F$ is positive on the characteristic set in $M$.
\item \label{claim: trapped set} The trapped set in $M$ is 
\[
    \Gamma = \bigcup_{(\xi_t, \xi_\phi) \in \R^2\backslash \{(0, 0)\}} \Gamma_{\xi_t, \xi_\phi},
\]
where
\[  
    \Gamma_{\xi_t, \xi_\phi} := \{\xi_r = 0, r = r_{\xi_t, \xi_\phi} \} \cap \Char(P).
\]
\item \label{claim: smooth trapping} $\Gamma$ is a smooth connected $5$-dimensional submanifold of $T^*M$, with defining functions $\xi_r, r - r_{\xi_t, \xi_\phi}$ and $\p$.
\item \label{claim: linearization} The linearization of the bicharacteristic flow at $\Gamma$ is given by
\begin{align*}
    \H_\p \begin{pmatrix} r - r_{\xi_t, \xi_\phi} \\ \xi_r \end{pmatrix}
        &= \frac1{(r_{\xi_t, \xi_\phi}^2 + a^2\cos^2(\theta))} \begin{pmatrix} 0 & 2 \mu(r_{\xi_t, \xi_\phi}) \\
        b^2F''(r_{\xi_t, \xi_\phi}) & 0 \end{pmatrix} \begin{pmatrix} r - r_{\xi_t, \xi_\phi} \\ \xi_r \end{pmatrix} \\
        &\qquad + \O\left((r - r_{\xi_t, \xi_\phi})^2 + \xi_r^2 \right)
\end{align*}
\end{enumerate}
In particular, the trapping in the domain of outer communication in
any subextremal Kerr-de Sitter spacetime is normally hyperbolic
trapping in the sense of \cite{WZ2011}. The stable ($s,-$) and unstable ($u,+$)
manifolds are the smooth manifolds given by
\[
	\Gamma^{u/s}=
		\left \{ 
			\xi_r=\pm \sgn(r-r_{\xi_t,\xi_\phi})b\sqrt{\frac{F(r)-F(r_{\xi_t,\xi_\phi})}{\mu}}
		\right\} \cap \Char(P).
\]
\end{thm}
The main computation in the proof is the following:
\begin{prop}\label{prop: main}
We have
\[
	h(r) := 2 \mu \d_r (r\mu' - 4 \mu) - \mu'(r\mu' - 4 \mu) < 0,
\]
for all $r \in (r_e, r_c)$.
\end{prop}

We postpone the proof of the Proposition \ref{prop: main}, as it will take up a significant amount of this section.

\begin{proof}[Proof of Theorem \ref{thm: trapping in Kerr-de Sitter}, assuming Proposition \ref{prop: main}]
We begin by proving claim \ref{claim: uniqueness}.
If $\xi_t = 0$, then the claim is clearly true with $r_{\xi_t, \xi_\phi} = r_0$, we may therefore assume that $\xi_t \neq 0$.
We now consider critical points of $F$ in $(r_e, r_c)$.
Defining
\[
    f(r) := \left((r^2 + a^2)\xi_t + a \xi_\phi\right)\mu' - 4 r \xi_t \mu
\]
we note that
\[
    F'(r) = - \frac{(r^2 + a^2)\xi_t + a \xi_\phi}{\mu^2} f(r),
\]
which vanishes if either 
\begin{equation} \label{eq: first condition}
	(r^2 + a^2)\xi_t + a \xi_\phi = 0
\end{equation}
or
\begin{equation} \label{eq: second condition}
	f(r) = 0.
\end{equation}
for some $r \in (r_e, r_c)$.
We claim that all critical points of $F$ in $(r_e, r_c)$ are local strict minima.
Note first that since $-4 r \xi_t \mu(r) \neq 0$, for all $r \in (r_e, r_c)$, there is no critical point of $F$ satisfying both \eqref{eq: first condition} and \eqref{eq: second condition} simultaneously.
Any critical point $r \in (r_e, r_c)$  of $F$, satisfying \eqref{eq: first condition}, satisfies
\[
	F''(r) 
		= - \frac{2r\xi_t}{\mu^2}\left( - 4 r \xi_t \mu \right)
		= 8 \frac{r^2\xi_t^2}{\mu} > 0,
\]
and is therefore a local strict minima.
It thus remains to study the case \eqref{eq: second condition}, which is equivalent to
\[
	(r^2 + a^2)\xi_t + a \xi_\phi = \frac{4r \mu \xi_t}{\mu'},
\]
since $\mu' \neq 0$ at the roots of $f$, since $- 4 \xi_t \mu \neq 0$.
At such points, we compute that
\begin{align*}
	f'(r) 
		&= \mu''(r)\left((r^2 + a^2)\xi_t + a \xi_\phi \right) + 2 r \mu'(r)\xi_t - 4 \mu(r)\xi_t - 4 r \mu'(r) \xi_t \\
		&= \xi_t\left( \mu''(r) \frac{4r \mu(r)}{\mu'(r)} - 4 \mu(r) - 2 r \mu'(r) \right) \\
		&= \frac {2\xi_t}{\mu'(r)} \left( 2 r \mu(r) \mu''(r) - 2 \mu(r) \mu'(r) - r \mu'(r)^2 \right) \\
		&= \frac {2\xi_t} {\mu'(r)} \left( 2 \mu(r) \d_r (r \mu'(r) - 4 \mu(r)) - \mu'(r)(r \mu'(r) - 4 \mu(r)) \right) \\
		&= \frac {2\xi_t} {\mu'(r)} h(r),
\end{align*}
where $h$ was defined in Proposition \ref{prop: main}.
It follows that all critical points of $F$, satisfying \eqref{eq: second condition}, satisfy
\begin{align*}
    F''(r)
        &= - \frac{(r^2 + a^2)\xi_t + a \xi_\phi}{\mu^2} f'(r)
        = - 8 r \frac{\xi_t^2}{\mu \mu'^2} h(r)
        > 0,
\end{align*}
by Proposition \ref{prop: main}.
We conclude that all critical points $r \in (r_e, r_c)$ of $F$ are local strict minima.
Now, if $(r^2 + a^2)\xi_t + a \xi_\phi$ vanishes precisely at $r = r_e$ or $r_c$, then $F \to \infty$ at the other end point of $[r_e, r_c]$ and it follows that $F$ cannot have critical points for they would all be local strict minima.
In the remaining case, $F \to \infty$ as $r \to r_e$ and $r_c$ and we conclude that there is a unique strict minimum $r_{\xi_t, \xi_\phi} \in (r_e, r_c)$.
This proves claim \ref{claim: uniqueness}.

We continue by proving claim \ref{claim: positivity}, i.e.\ that 
\[
	(r^2 + a^2)\xi_t + a \xi_\phi \neq 0
\]
in the characteristic set.
The principal symbol $\p$ of $P$ is given by
\begin{align*}
	(r^2 + a^2\cos^2(\theta)) \p(\xi)
		&= \mu(r) \xi_r^2 + c(\theta) \xi_\theta^2 + \frac{b^2}{c(\theta)\sin^2(\theta)} \left(a\sin^2(\theta)\xi_t + \xi_\phi \right)^2 \\*
		&\qquad - \frac{b^2}{\mu(r)} \left( (r^2 + a^2) \xi_t + a \xi_\phi \right)^2.
\end{align*}
Assume there is a point in the characteristic set satisfying $(r^2 + a^2)\xi_t + a \xi_\phi = 0$, then
\[
	a \sin^2(\theta) \xi_t + \xi_\phi = 0
\]
and we get a solution to the linear equation
\[
	\begin{pmatrix}
	r^2 + a^2 & a \\
	a \sin^2(\theta) & 1
	\end{pmatrix}
	\begin{pmatrix}
	\xi_t \\
	\xi_\phi
	\end{pmatrix}
	=
	0,
\]
and the determinant of the matrix is
\[
	r^2 + a^2 \cos^2(\theta) > 0.
\]
This implies that $\xi_t = \xi_\phi = 0$, which in turn implies that $\xi_r = \xi_\theta = 0$.

Let us now show claim \ref{claim: trapped set}.
For this, recall first that $\Char(P)$ is invariant under $\H_\p$.
Since the bicharacteristic flow is invariant under conformal changes, let us for simplicity study
\begin{align*}
	\q(\xi)
	    &=(r^2 + a^2\cos^2(\theta)) \p(\xi) \\
		&= \mu(r) \xi_r^2 + c(\theta) \xi_\theta^2 + \frac{b^2}{c(\theta)\sin^2(\theta)} \left(a\sin^2(\theta)\xi_t + \xi_\phi \right)^2 \\*
		&\qquad - \frac{b^2}{\mu(r)} \left( (r^2 + a^2) \xi_t + a \xi_\phi \right)^2.
\end{align*}
Since $\H_\q \xi_t = \H_\q \xi_\phi = 0$, it follows that $\H_\q r_{\xi_t, \xi_\phi} = 0$ and we use this to compute
\begin{align}
	\H_\q (r - r_{\xi_t, \xi_\phi})
	    &= 2 \mu(r) \xi_r, \label{eq: Hamiltonian r} \\
	\H_\q \xi_r
		&= - \d_r \left( \mu \xi_r^2 - \frac{b^2}{\mu(r)} \left( (r^2 + a^2) \xi_t + a \xi_\phi \right)^2 \right) \nonumber \\
		&= - \mu' \xi_r^2 + b^2 F'(r). \label{eq: Hamiltonian xi_r}
\end{align}
Evaluating these at $\Gamma$, both expressions vanish and it follows that $\Gamma$ is invariant under $\H_\q$.
We claim that
\begin{equation} \label{eq: convexity}
    x \notin \Gamma, \quad \H_\q \left( r - r_{\xi_t, \xi_\phi} \right)^2|_x = 0 \quad \Longrightarrow \quad \left(\H_\q\right)^2 \left( r - r_{\xi_t, \xi_\phi} \right)^2|_x > 0.
\end{equation}
To prove this, assume that
\begin{align*}
    0
        = \H_\q \left( r - r_{\xi_t, \xi_\phi} \right)^2|_x
        = 2 \left( r - r_{\xi_t, \xi_\phi} \right) \H_\q r|_x,
\end{align*}
which implies that either $r = r_{\xi_t, \xi_\phi}$ or $\H_\q r = 0$ at $x$.
We compute
\[
    \left(\H_\q\right)^2 \left( r - r_{\xi_t, \xi_\phi} \right)^2|_x
        = 2 \left( \H_\q r \right)^2|_x + 2 \left( r - r_{\xi_t, \xi_\phi} \right) \left(\H_\q\right)^2r|_x.
\]
Now, if $r|_x = r_{\xi_t, \xi_\phi}$, we see that 
\[
    \left(\H_\q\right)^2 \left( r - r_{\xi_t, \xi_\phi} \right)^2|_x 
        = 2 \left( \H_\q r \right)^2|_x \geq 0
\]
and it vanishes if and only if $\H_\q r|_x = 0$, in which case $\xi_r|_x = 0$ and hence $x \in \Gamma$.
Similarly, if instead $\H_\q r|_x = 0$, then $\xi_r|_x = 0$ and we conclude that
\begin{align*}
    \left(\H_\q\right)^2 \left( r - r_{\xi_t, \xi_\phi} \right)^2|_x
        &= 2 \left( r - r_{\xi_t, \xi_\phi} \right) \left(\H_\q\right)^2r|_x \\
        &= 4 \left( r - r_{\xi_t, \xi_\phi} \right) \mu(r) b^2 F'(r)|_x.
\end{align*}
Recalling that $r - r_{\xi_t, \xi_\phi}$ and $F'(r)$ have the same sign, we conclude that this is positive unless $r = r_{\xi_t, \xi_\phi}$, in which case $x \in \Gamma$.
This proves claim \eqref{eq: convexity}.
For any $C > 0$, we define the function
\[
    \E := e^{C\left(r - r_{\xi_t, \xi_\phi}\right)^2} \H_\q \left(r - r_{\xi_t, \xi_\phi} \right)^2.
\]
We get
\[
    \H_\q \E = e^{C\left(r - r_{\xi_t, \xi_\phi}\right)^2} \left(C \left(\H_\q \left(r - r_{\xi_t, \xi_\phi} \right)^2\right)^2 + \left(\H_\q\right)^2 \left(r - r_{\xi_t, \xi_\phi} \right)^2 \right),
\]
which by \eqref{eq: convexity} is positive precisely away from $\Gamma$ if $C$ is large enough.
Thus $\E$ provides a globally defined escape function which grows along each bicharacteristic not in $\Gamma$ and vanishes identically at $\Gamma$.
Note that $\E$ is also an escape function in the case when $(r^2 + a^2)\xi_t + a \xi_\phi$ vanishes at $r_e$ or $r_c$, by substituting $r_{\xi_t, \xi_\phi}$ in the expression for $\E$ with that point.
This proves in particular that $\Gamma$ is precisely the trapped set in $M$, which is claim \ref{claim: trapped set}.

For claim \ref{claim: smooth trapping}, we need to prove that
\[
    \md \left(r - r_{\xi_t, \xi_\phi}\right), \ \md \xi_r \text{ and } \md \p
\]
are linearly independent at $\Gamma$. 
This follows by noting that
\[
    \md \p \left( \frac{\d}{\d r} \right)|_\Gamma
        = \frac{\d \p}{\d r}|_\Gamma
        = \frac1{r^2 + a^2 \cos^2(\theta)} \left( \mu' \xi_r^2 - b^2 F' \right)|_\Gamma 
        = 0.
\]

Finally, the linearization claim in \ref{claim: linearization} follows
immediately from \eqref{eq: Hamiltonian r} and \eqref{eq: Hamiltonian xi_r} by Taylor expanding $F'$ around $r = r_{\xi_t, \xi_\phi}$. 

The full stable and unstable submanifolds then are given as in Dyatlov's
paper \cite{Dya2015}*{Proposition~3.5}. 
By equation \eqref{eq: Hamiltonian xi_r}, we note that
\[
	\H_\q \left( \mu \xi_r^2 - b^2 \left( F - F\left(r_{\xi_t, \xi_\phi}\right)\right) \right) = 0,
\]
which shows that the flow of $\H_\q$ leaves $\Gamma^{u/s}$ invariant.
Moreover, note that
\[
	\Gamma = \Gamma^u \cap \Gamma^s.
\]
The defining functions for $\Gamma^{u/s}$ are
\[
	\xi_r \mp \sgn(r-r_{\xi_t,\xi_\phi})b\sqrt{\frac{F(r)-F(r_{\xi_t,\xi_\phi})}{\mu}} \quad \text{and } \p,
\]
and since $\md \p(\d_{\xi_r})|_\Gamma = 0$, it follows that $\Gamma^{u/s}$ are smooth submanifolds of $\Char(P)$ near $\Gamma$.
A simple rewriting taking into account the eigenvectors of the linearization, or indeed the sign of the escape function $\E$, namely negative on the stable (so $\sgn\xi_r=-\sgn(r-r_{\xi_t,\xi_\phi})$ there), positive on the unstable manifold (so $\sgn\xi_r=\sgn(r-r_{\xi_t,\xi_\phi})$ there), gives the conclusion.
\end{proof}

In order to prove Proposition \ref{prop: main}, we need the following two lemmas.
The first one gives three useful ways of rewriting $h$:

\begin{lemma} \label{le: the polynomial}
We have
\begin{align}
	h(r)
		&= - \frac1r (r\mu' - 4\mu)^2 - 4 \mu \left( 3m - \frac{4 a^2}r \right) \label{eq: Andras identity} \\
		&= 4 \Lambda mr^4 - 4 b^2 r^3 + 12m \left( 1 - \frac{\Lambda a^2}3 \right) r^2 - 12 m^2 r + 4ma^2 \label{eq: full expression} \\
			\begin{split}
				&=\frac {4\Lambda r^4}3 \left( 3 m - \frac{4a^2}r \right) - 4 \left( 1 - \frac{\Lambda a^2}3 \right)^2 \left( r - \frac m{1 - \frac{\Lambda a^2}3}\right)^3 \\
				&\qquad + 4 m^3 \left( \frac{a^2}{m^2} - \frac1{1 - \frac{\Lambda a^2}3} \right),
			\end{split}\label{eq: new identity}
\end{align}
for all $r \in \R$.
\end{lemma}
\begin{remark} \label{rmk: Lambda a2}
Note that the discriminant condition \eqref{eq: discriminant} ensures that
\[
    1 - \frac{\Lambda a^2}3 > 0,
\]
hence the expressions in the lemma make sense.
\end{remark}

\begin{proof}
Note that
\begin{align*}
	r \mu' - 4 \mu 
		&= - 2 \left( 1 - \frac{\Lambda a^2}3 \right) r^2 + 6m r - 4a^2, \\
	\d_r(r \mu' - 4 \mu)
		&= - 4 \left( 1 - \frac{\Lambda a^2}3 \right) r + 6m \\
		&= \frac 2r \left( r \mu' - 4 \mu \right) - 6m + \frac{8 a^2}r.
\end{align*}
Inserting this gives 
\begin{align*}
	2 \mu \d_r (r \mu' - 4 \mu) - \mu'(r \mu' - 4 \mu)
		&= 2 \mu \left(\frac 2r \left( r \mu' - 4 \mu \right) - 6m + \frac{8 a^2}r \right) - \mu'(r \mu' - 4 \mu) \\
		&= (r\mu' - 4\mu)\left( \frac{4 \mu}r - \mu' \right) - 4 \mu \left( 3m - \frac{4 a^2}r \right) \\
		&= - \frac1r (r\mu' - 4\mu)^2 - 4 \mu \left( 3m - \frac{4 a^2}r \right),
\end{align*}
proving identity \eqref{eq: Andras identity}.
The identity \eqref{eq: full expression} now follows by
\begin{align*}
	&- \frac1r (r\mu' - 4\mu)^2  - 4 \mu \left( 3m - \frac{4 a^2}r \right) \\*
		&\quad = - \frac1r \left( - 2 \left( 1 - \frac{\Lambda a^2}3 \right) r^2 + 6mr - 4a^2 \right)^2 \\
		&\quad \qquad - 4 \left( - \frac{\Lambda r^4}3 + \left( 1 - \frac{\Lambda a^2}3 \right) r^2 - 2mr + a^2 \right) \left( 3m - \frac{4 a^2}r \right) \\*
		&\quad = - 4 \left( 1 - \frac{\Lambda a^2}3 \right)^2 r^3 - 36 m^2 r - \frac{16 a^4}r \\
		&\quad \qquad + 24 m \left( 1 - \frac{\Lambda a^2}3 \right) r^2 - 16 \left( 1 - \frac{\Lambda a^2}3 \right) a^2 r + 48 ma^2 \\
		&\quad \qquad + 4 \Lambda m r^4 - 12 m \left( 1 - \frac{\Lambda a^2}3 \right) r^2 + 24m^2 r - 12 ma^2 \\
		&\quad \qquad - \frac{16 \Lambda a^2}3 r^3 + 16 \left( 1 - \frac{\Lambda a^2}3 \right) a^2 r - 32 ma^2 + \frac{16a^4}r \\
		&\quad = 4 \Lambda mr^4 - 4 \left( 1 + \frac{\Lambda a^2}3 \right)^2 r^3 + 12 m \left( 1 - \frac{\Lambda a^2}3 \right) r^2 - 12 m^2r + 4 ma^2,
\end{align*}
as claimed.
Finally, the identity \eqref{eq: new identity} follows by
\begin{align*}
	&4 \Lambda mr^4 - 4 \left( 1 + \frac{\Lambda a^2}3 \right)^2 r^3 + 12 m \left( 1 - \frac{\Lambda a^2}3 \right) r^2 - 12 m^2r + 4 ma^2 \\
	&\quad = 4 \Lambda mr^4 - 4 \left( 1 + \frac{\Lambda a^2}3 \right)^2 r^3 + 4 \left( 1 - \frac{\Lambda a^2}3 \right)^2 r^3 \\*
	&\quad \qquad - 4 \left( 1 - \frac{\Lambda a^2}3 \right)^2 r^3 + 12 m \left( 1 - \frac{\Lambda a^2}3 \right) r^2 - 12 m^2 r + \frac{4 m^3}{1 - \frac{\Lambda a^2}3} \\*
	&\quad \qquad + 4m^3 \left( \frac{a^2}{m^2} - \frac1{1 - \frac{\Lambda a^2}3} \right) \\
	&\quad = \frac {4\Lambda r^4}3 \left( 3 m - \frac4r a^2 \right) - 4 \left( 1 - \frac{\Lambda a^2}3 \right)^2 \left( r - \frac m{1 - \frac{\Lambda a^2}3}\right)^3 \\*
	&\quad \qquad + 4 m^3 \left( \frac{a^2}{m^2} - \frac1{1 - \frac{\Lambda a^2}3} \right). \qedhere
\end{align*}
\end{proof}

\begin{lemma} \label{le: discriminant}
Given $a \in \R$ and $m > 0$, there is a (potentially empty) interval $(\Lambda_0, \Lambda_1)$ such that $a, m$ and $\Lambda \geq 0$ satisfy the discriminant condition \eqref{eq: discriminant} if and only if
\[
    \Lambda \in (\Lambda_0, \Lambda_1),
\]
in which case we have
\begin{equation} \label{eq: maximal condition}
    \frac{a^2}{m^2} < \frac98\frac1{1 - \frac{\Lambda a^2}3}.
\end{equation}
If $\abs a > m$, then $\Lambda_0 > 0$.
\end{lemma}
\begin{proof}
Writing $\gamma := \frac{\Lambda a^2}3$, the discriminant condition \eqref{eq: discriminant} becomes
\[
    - \left(1 + \gamma\right)^4 \frac{a^2}{m^2} + 36(1 - \gamma)\gamma + (1 - \gamma)^3 - 27\gamma \frac{m^2}{a^2} > 0.
\]
We introduce
\[
    \nu := \frac \gamma {(1 - \gamma)^2}, \quad \a := \frac{m^2}{a^2(1 - \gamma)}.
\]
Multiplying with $\frac{m^2}{a^2}$ and using that
\[
    (1 + \gamma)^4 = ((1 - \gamma)^2 + 4 \gamma)^2 = \frac{\gamma^2}{\nu^2} + 8 \frac{\gamma^2}\nu + 16 \gamma^2,
\]
we note that \eqref{eq: discriminant} is equivalent to
\[
    \frac{\gamma^2}{\nu^2}\left( - 1 - 8 \nu - 16 \nu^2 + 36 \nu \a + \a - 27 \nu \a^2\right) > 0
\]
which in turn is equivalent to
\[
    - 16 \nu^2 + (- 27 \a^2 + 36 \a - 8) \nu - 1 + \a > 0.
\]
Since this is a quadratic expression in $\nu$ and since
\[
    (0, 1) \ni \gamma \mapsto \nu
\]
is injective, this proves the first assertion by Remark \ref{rmk: Lambda a2}.
The discriminant of the quadratic expression in $\nu$ is
\[
    \a (9 \a - 8)^3,
\]
which is positive if and only if \eqref{eq: maximal condition} is satisfied.
Finally, in case $\abs a > m$, equation \eqref{eq: discriminant} implies that $\Lambda = 0$ is not allowed, proving the last assertion.
\end{proof}

\begin{remark}
The case of Reissner-Nordstr\"om-de Sitter spacetimes gives rise to a
very similar subextremality condition, which was explicitly analyzed
in \cite{H2018}*{Proposition~3.2}, following \cite{KI2004}.
  \end{remark}

\begin{proof}[Proof of Proposition \ref{prop: main}] \

\noindent
\textbf{Step 1:}
\emph{The case when $r \in \left( \frac{4a^2}{3m}, r_c \right) \cap \left( r_e, r_c \right)$.}
This is immediate from \eqref{eq: Andras identity}, by noting that
\[
    3 m - \frac{4a^2}r
		> 0, \quad \mu(r) > 0,
\]
in the interval.

\noindent
\textbf{Step 2:}
\emph{The case when $\frac1{1 - \frac{\Lambda a^2}3} \geq \frac{a^2}{m^2}$.}
Since
\[
	\mu'(r) = - \frac{4\Lambda r^3}3 + 2 \left( 1 - \frac{\Lambda a^2}3 \right) r - 2m,
\]
we know that
\[
	\mu'(r) < 0,
\]
for all 
\[
	r \in \left[0, \frac m{1 - \frac{\Lambda a^2}3} \right].
\]
Since we know that
\[
	r_- < 0 < r_C < r_e < r_c,
\]
and
\[
	\mu(0) = a^2 > 0,
\]
it follows that
\[
	r_e > \frac m{1 - \frac{\Lambda a^2}3}.
\]
The proof in this case is completed by combining this with \eqref{eq: new identity} and Step 1.

\noindent
\textbf{Step 3:}
\emph{The remaining case.}
We are left with proving the statement when
\begin{align*}
	&\frac1{1 - \frac{\Lambda a^2}3} 
		< \frac{a^2}{m^2}, \\
	&r 
		\in \left( r_e, \frac{4a^2}{3m} \right].
\end{align*}
The assertion in Proposition \ref{prop: main} is the negativity of
\[
	h(r) := 2 \mu f' - \mu' f,
\]
where
\[
	f(r) := r \mu' - 4 \mu = - 2 \left( \left( 1 - \frac{\Lambda a^2}3 \right) r^2 - 3mr + 2 a^2 \right).
\]
By Lemma \ref{le: discriminant}, we need to check the condition in an interval $\Lambda \in (\Lambda_0, \Lambda_1)$, where \[
    \Lambda_0 > 0.
\]
We first claim that $h(r)$ decreases with increasing $\Lambda$ for all $r \in \left( r_e, \frac{4 a^2}{3m}\right)$.
Differentiating $h$ with respect to $\Lambda$, using equation \eqref{eq: full expression}, gives
\[
	\frac{\d}{\d \Lambda} h(r)
		= 4 m r^4 - \frac 83 a^2br^3 - 4 m a^2r^2 = 4 m^3r^2 \left(\left(\frac r m\right)^2 - \frac 23 \frac{a^2}{m^2} b \frac rm - \frac{a^2}{m^2}\right),
\]
which is negative for all $r \in (r_1, r_2)$, where
\[
	\frac{r_{1/2}}m = \frac{a^2b}{3 m^2} \mp \sqrt{\frac{a^4b^2}{9 m^4} + \frac{a^2}{m^2}}.
\]
The claim is proven if we can show that $(r_1, r_2) \subseteq (r_e, r_c)$.
Since
\[
    r_1 < 0 < r_e,
\]
it remains to show that
\[
    r_2 \geq \frac{4a^2}{3m},
\]
which is equivalent to
\[
    \sqrt{\frac{a^4}{9 m^4}\left(1 + \frac{\Lambda a^2}3\right)^2 + \frac{a^2}{m^2}} 
        \geq \frac{4a^2 - a^2 b}{3m^2}
        = \frac{3 - \frac{\Lambda a^2}3}{3m^2}a^2,
\]
which in turn is equivalent to
\begin{align*}
    \frac{a^2}{m^2}
        &\geq \left(\left(3 - \frac{\Lambda a^2}3\right)^2 - \left(1 + \frac{\Lambda a^2}3\right)^2\right) \frac{a^4}{9 m^4} \\
        &= 8 \left(1 - \frac{\Lambda a^2}3\right)\frac{a^4}{9 m^4}.
\end{align*}
This inequality is equivalent to \eqref{eq: maximal condition}, so the proof of the claim is complete.

It therefore suffices to prove the negativity of $h(r)$ for all $r \in \left( r_e, \frac{4 a^2}{3m}\right]$, in the limit when $\Lambda = \Lambda_0$.
By definition of $\Lambda_0$, $\mu$ has at least a double root when $\Lambda = \Lambda_0$.
Since $\mu(0) = a^2 > 0$, independent of $\Lambda$, there is still a simple negative root $r_- < 0$.
Since $\mu$ decreases with $\Lambda$, the case $r_C < r_e = r_c$ is excluded for $\Lambda = \Lambda_0$, for it would contradict that $\mu$ has four distinct real roots for any $\Lambda > \Lambda_0$.
Similarly, the case $r_C = r_e = r_c$ can be excluded, since also $\mu'$ decreases with $\Lambda$ and would in that case not have three distinct real roots for any $\Lambda > \Lambda_0$.
We conclude that $r_- < r_C = r_e < r_c$ when $\Lambda = \Lambda_0$.
In this case, $\mu(r_e) = \mu'(r_e) = 0$ and we note that $r_e$ satisfies the equation
\begin{equation} \label{eq: extremal polynomial}
	0 = r_e \mu'(r_e) - 4 \mu(r_e) = - 2 \left( \left( 1 - \frac{\Lambda_0 a^2}3 \right) r_e^2 - 3mr_e + 2 a^2 \right).
\end{equation}
Since
\[
	r\mu' - 4 \mu = r^5\d_r \left( r^{-4} \mu \right),
\]
and since $r^{-4} \mu$ has two positive zeros $r_C = r_e$ and $r_c$, we conclude that $r_e$ coincides with the smaller root of \eqref{eq: extremal polynomial}:
\[
	r_e = \frac{3m}{2\left( 1 - \frac{\Lambda_0 a^2}3 \right)} - \frac m{2\left( 1 - \frac{\Lambda_0 a^2}3 \right)} \sqrt{9 - 8 \frac{a^2}{m^2}\left( 1 - \frac{\Lambda_0 a^2}3\right)}.
\]
The idea is now to estimate the Taylor expansion at $r_e$ for $r \geq r_e$. 
We write
\[
    h(r) 
        = \sum_{k = 0}^4 h^{(k)}(r_e) \frac{(r - r_e)^k}{k!}.
\]
It is immediate from \eqref{eq: full expression} that
\[
    \frac{h^{(4)}(r_e)}{4!} = 4 \Lambda_0 m.
\]
Using that
\[
	\mu(r_e) = \mu'(r_e) = f(r_e) = 0,
\]
one computes that
\begin{align*}
    h(r_e) 
        &= h'(r_e) = h''(r_e) = 0, \\
	h'''(r_e)
		&= 3\mu''(r_e)f''(r_e) - \mu'''(r_e)f'(r_e).
\end{align*}
We note that
\[
	f'(r) = - 4 \left( 1 - \frac{\Lambda_0 a^2}3 \right) r + 6m = r \mu''(r) - 3 \mu'(r)
\]
and hence
\[
	f'(r_e) = 2m \sqrt{9 - 8 \frac{a^2}{m^2}\left( 1 - \frac{\Lambda_0 a^2}3 \right)} = r_e \mu''(r_e).
\]
Moreover,
\[
	f''(r_e) = r_e \mu'''(r_e) - 2 \mu''(r_e).
\]
Inserting this, we get
\begin{align}\label{eq:hppp-estimate}
	h'''(r_e)
		&= - 6 \mu''(r_e)^2 + 2 r_e \mu''(r_e) \mu'''(r_e) \\
		&\leq 2 f'(r_e) \mu'''(r_e) \\
		&= - 32 \Lambda_0 m r_e \sqrt{9 - 8 \frac{a^2}{m^2}\left( 1 - \frac{\Lambda_0 a^2}3 \right)}.
\end{align}
Using that
\[
	r_e 
		\geq \frac m{1 - \frac{\Lambda_0 a^2}3},
\]
the Taylor expansion at $r = r_e$ can be estimated as
\begin{align*}
	\frac{h(r)}{( r - r_e )^3}
		&\leq - \frac{16}3 \frac{\Lambda_0 m^2}{1 - \frac{\Lambda a^2}3} \sqrt{9 - 8 \frac{a^2}{m^2}\left( 1 - \frac{\Lambda_0 a^2}3 \right)} + 4 \Lambda_0 m (r - r_e)
\end{align*}
for $r \geq r_e$.
Now, the above bound is increasing with $r$, it therefore suffices to check the negativity at
\[
	\hat r 
	    := \frac{3m}{2 \left( 1 - \frac{\Lambda_0 a^2}3 \right)}
	     \geq \frac 32 \frac 89 \frac{a^2}{m} 
        = \frac{4 a^2}{3 m}.
\]
We get
\begin{align*}
	\frac{h\left( \hat r \right)}{\left( \hat r - r_e \right)^3}
		&\leq - \frac{16}3 \frac{\Lambda_0 m^2}{1 - \frac{\Lambda a^2}3} \sqrt{9 - 8 \frac{a^2}{m^2}\left( 1 - \frac{\Lambda_0 a^2}3 \right)} \\*
		&\qquad + 2 \frac{\Lambda_0 m^2}{1 - \frac{\Lambda_0 a^2}3} \sqrt{9 - 8 \frac{a^2}{m^2}\left( 1 - \frac{\Lambda_0 a^2}3 \right)} \\
		&< 0.
\end{align*}
We have in particular shown that $h(r) < 0$ when  
\[
	r \in \left( r_e, \frac{4a^2}{3m} \right], \quad \Lambda = \Lambda_0.
\]
It follows that $h < 0$ on this interval for all $\Lambda \in (\Lambda_0, \Lambda_1)$.
The proof is completed by applying Step 1.
\end{proof}

\begin{remark}
  We also present a somewhat different approach to proving
  \eqref{eq:hppp-estimate}, which was the key computation above.
  By \eqref{eq: full expression} we have, with $\Lambda=\Lambda_0$, $r=r_e$,
$$
h'''(r)=24 (4\Lambda m r-(1+\gamma)^2),
$$
and want to prove that
\begin{equation}\label{eq:hppp-est2}
	h'''(r_e)
		\leq -96\Lambda r_e\sqrt{1-\frac{8}{9\alpha}},
\end{equation}
where, as in Lemma~\ref{le: discriminant}, $\alpha=\frac{m^2}{a^2(1-\gamma)}$, $\gamma=\frac{\Lambda a^2}{3}$, and $\alpha\in[8/9,1]$ for us. Also,
$$
	r_e=\frac{3m}{2(1-\gamma)}\Big(1-\sqrt{1-\frac{8}{9\alpha}}\Big).
$$
Thus,
$$
	h'''(r_e)=24 \left(\frac{6\Lambda m^2}{1-\gamma}\Big(1-\sqrt{1-\frac{8}{9\alpha}}\Big)-(1+\gamma)^2\right),
$$
and
\begin{equation}\label{eq:Lm2-exp}
\Lambda m^2=3\frac{\Lambda a^2}{3}\frac{m^2}{a^2(1-\gamma)}(1-\gamma)=3\alpha\gamma(1-\gamma),
\end{equation}
so
$$
	\frac{h'''(r_e)}{24}=18\alpha\gamma\Big(1-\sqrt{1-\frac{8}{9\alpha}}\Big)-(1+\gamma)^2.
$$
On the other hand, the right hand side of \eqref{eq:hppp-est2} divided by 24 is
$$
-4\frac{3\Lambda m^2}{2(1-\gamma)}\Big(1-\sqrt{1-\frac{8}{9\alpha}}\Big)\sqrt{1-\frac{8}{9\alpha}}.
$$
which is, using the rewriting \eqref{eq:Lm2-exp} for $\Lambda m^2$,
$$
-18\alpha\gamma\Big(1-\sqrt{1-\frac{8}{9\alpha}}\Big)\sqrt{1-\frac{8}{9\alpha}}.
$$
Subtracting this from $h'''(r_e)/24$, one wants to show that it is $\leq 0$. But this is
$$
18\alpha\gamma\Big(1-\sqrt{1-\frac{8}{9\alpha}}\Big)\Big(1+\sqrt{1-\frac{8}{9\alpha}}\Big)-(1+\gamma)^2,
$$
which simplifies to
$$
	18\gamma\alpha \left( 1-1+\frac{8}{9\alpha} \right) -(1+\gamma)^2=16\gamma-(1+\gamma)^2,
$$
which is exactly the negative of the quadratic polynomial one gets in the discriminant when one wants to assure that $r\mu'-\mu=r^2(r^{-1}\mu)'$ has distinct real roots, as is the case in the
subextremal range, namely one each in $(r_C,r_e)$ and in $(r_e,r_c)$
corresponding to the critical points of $r^{-1}\mu$ there, i.e.\ $\gamma$ has to be such that the negative of this
is $\geq 0$, so it is indeed $\leq 0$.
\end{remark}

Finally, we combine Theorem \ref{thm: trapping in Kerr-de Sitter} with the results of \cite{V2013} in the proof of Theorem \ref{thm: asymptotic expansion}:
\begin{proof}[Proof of Theorem \ref{thm: asymptotic expansion}]
We again consider the operator
\[
    P_\s: \{u \in \bar H^s \mid P_\s u \in \bar H^{s-1} \} \to \bar H^{s-1}
\]
from Theorem \ref{thm: Fredholm}.
The semi-classical trapping is corresponding to the trapping of bicharacteristics of the full wave operator $P$.
Given that Theorem \ref{thm: trapping in Kerr-de Sitter} implies that the trapping is normally hyperbolic, the proof of the semi-classical estimates and consequently the proof of Theorem \ref{thm: asymptotic expansion} proceeds completely analogous to the proof of \cite{V2013}*{Theorem~1.4}.
\end{proof}

\section*{Acknowledgements}
\noindent
The authors are grateful to Peter Hintz, Rita Texeira da Costa and Maciej Zworski for very helpful discussions and comments on the manuscript. 
The first author is grateful for the support of the Deutsche Forschungsgemeinschaft (DFG, German Research Foundation) – GZ: LI3638/1-1, AOBJ: 660735 and the Stanford Mathematics Research Center. 
The second author gratefully acknowledges support from the National Science Foundation under grant number DMS-1953987.

\end{sloppypar}
\end{document}